\documentclass[12pt, a4paper, reqno, oneside]{article}

\usepackage{amssymb,amsmath}
\usepackage{amsmath,accents}
\DeclareUnicodeCharacter{0335}{}
\usepackage{tikz-cd}
\usepackage{authblk} 
\usepackage{physics} 
\usepackage{tikz}
\usepackage{titlesec} 
\usetikzlibrary{babel}
\usepackage{adjustbox}
\usepackage{url}
\usepackage{xcolor}
\usepackage{yhmath}
\usepackage{eso-pic}
\usepackage[T1]{fontenc}
\usepackage[utf8]{inputenc}
\usepackage[spanish,english]{babel}
\usepackage{csquotes}
\usepackage{pst-node}
\usepackage[english]{babel}
\usepackage{latexsym}
\usepackage{euscript}
\usepackage{mathtools}
\usepackage[all]{xy}
\usepackage{mathrsfs}
\usepackage[percent]{overpic}
\usepackage{setspace}
\usepackage{graphicx}
\usepackage{pdfsync}
\usepackage{amsthm}
\usepackage{thmtools}
\usepackage{enumitem}
\usepackage{yfonts}
\usetikzlibrary{calc}
\usetikzlibrary{decorations.pathmorphing}

\tikzset{curve/.style={settings={#1},to path={(\tikztostart)
    .. controls ($(\tikztostart)!\pv{pos}!(\tikztotarget)!\pv{height}!270:(\tikztotarget)$)
    and ($(\tikztostart)!1-\pv{pos}!(\tikztotarget)!\pv{height}!270:(\tikztotarget)$)
    .. (\tikztotarget)\tikztonodes}},
    settings/.code={\tikzset{quiver/.cd,#1}
        \def\pv##1{\pgfkeysvalueof{/tikz/quiver/##1}}},
    quiver/.cd,pos/.initial=0.35,height/.initial=0}

\tikzset{tail reversed/.code={\pgfsetarrowsstart{tikzcd to}}}
\tikzset{2tail/.code={\pgfsetarrowsstart{Implies[reversed]}}}
\tikzset{2tail reversed/.code={\pgfsetarrowsstart{Implies}}}
\tikzset{no body/.style={/tikz/dash pattern=on 0 off 1mm}}

\usepackage[colorlinks=true,linktocpage=true, pagebackref=false, citecolor=black,linkcolor=black]{hyperref}
\normalbaselines
\makeatletter

\def\ps@myfancy{\let\@mkboth\markboth
 \def\@evenhead{\vbox{\hsize\textwidth 
 \hbox to \textwidth{\sf\mdseries\thepage 
 \rule[-.6ex]{0mm}{2mm} \hfill\sf\large\leftmark}
 \vskip 1pt \hrule}}
 \def\@oddhead{\vbox{\hsize\textwidth 
 \hbox to \textwidth{{\sf\large\leftmark}
 \rule[-.6ex]{0mm}{2mm} \hfill\sf\mdseries{\thepage}}
 \vskip 1pt \hrule}}}

\def\ps@myfancyplain{
 \def\@evenhead{\vbox{\hsize\textwidth%
 \rule[-.6ex]{0mm}{2mm} \hfill }
 \vskip 1pt \hrule
 \vskip\headsep
 \vskip\textheight
 \vskip1pc
 \hbox to \textwidth{\sf\mdseries\thepage 
 \rule[-6ex]{0mm}{2mm} \hfill }}
 \def\@oddhead{\vbox{\hsize\textwidth 
 \vskip 1pt\hrule
 \vskip\headsep
 \vskip\textheight
 \vskip2pc
 \hbox to \textwidth{\hfill\rule[.4ex]{1pc}{2.5pt}
 \sf\mdseries\thepage}
}}}

\def\ps@myemptyfun{
 \def\@evenhead{\vbox{\hsize\textwidth
 \rule[-.6ex]{0mm}{2mm} \hfill }
 \vskip 1pt 
 \vskip\headsep
 \vskip\textheight
 \vskip1pc
 \hbox to \textwidth{\sf\mdseries\thepage 
 \rule[-0.6ex]{0mm}{2mm} 
 \hfill }}
 \def\@oddhead{\vbox{\hsize\textwidth 
 \vskip 1pt 
 \vskip\headsep
 \vskip\textheight
 \vskip2pc
}}}

\makeatother
\providecommand{\proofname}{Demostraci\'on.}
 {\par\noindent{\it Demostraci\'on. }\nopagebreak\normalsize}%

 {\par\noindent{\it #1. }\nopagebreak\normalsize}%
 {\hfill\linebreak[2]\hspace*{\fill}$\square$\\[-1pt]}
\makeatother

\def\sqbullet{\raise.2ex\hbox{\vrule width 3.5pt height 3.5pt}}

{}

{\em}

\newcounter{substep}
\def\thesubstep{\arabic{substep}}

{\em}

\newcounter{subsubstep}
\def\thesubsubstep{\arabic{subsubstep}}

{\em}

\numberwithin{figure}{section}

{}

\newtheoremstyle{mystyle}
  {}
  {}
  {\itshape}
  {}
  {\sf \bfseries}
  {}
{ }
  {\thmname{#1}\thmnumber{{\textcolor{blue}{\, \hspace{-1mm}#2.}}}\thmnote{ (#3)}}

\theoremstyle{mystyle}
\definecolor{royalblue(web)}{rgb}{0.25, 0.41, 0.88}
\hypersetup{
    colorlinks=true,
    linkcolor=royalblue(web),
    filecolor=magenta,      
    urlcolor=cyan,
    pdftitle={Overleaf Example},
    pdfpagemode=FullScreen,
    }

\titleformat{\section}
  {\normalfont\LARGE\bfseries}{\thesection.}{1em}{}
\titleformat{\subsection}
  {\normalfont\Large\bfseries}{\thesubsection.}{1em}{}
\titleformat{\subsubsection}
  {\normalfont\normalsize\itshape}{\thesubsubsection.}{1em}{}

\newtheorem{Teor}{Theorem}[section]

\newtheorem{Prop}[Teor]{Proposition}
\newtheorem{Coro}[Teor]{Corollary}

\newtheorem{Defi}[Teor]{Definition}

\newtheorem{Lema}[Teor]{Lemma}

\newtheorem{Obses}[Teor]{Observations}

\newtheorem{Obse}[Teor]{Remark}

\newcommand{\R}{{\mathbb R}}

\newcommand{\mail}[1]{\small\href{mailto:#1}{#1}}
\newcommand{\CR}[1]{{\color{blue}\textbf{\textsf{[Obs: #1]}}}}

\newenvironment{Abstract}
{
\begin{center}
\textbf{Abstract}\\
\vspace{0.25cm}
\begin{minipage}{14.5cm}}
{\footnotesize
\end{minipage}
\end{center}}
\begin{document}

	
	\begin{center}
		{\huge {\bfseries Nonlocal gradient, the nonlocal Laplacian and maximum principles \par}}
		\vspace{1cm}
	\end{center}
	\begin{center} 
		\begin{tabular}{l@{\hskip 2cm}l} 
			{\Large Jos\'e C. Bellido }{\small\textsuperscript{1}} & {\Large Guillermo García-Sáez }{\small\textsuperscript{1}} \\
			\mail{josecarlos.bellido@uclm.es} & \mail{guillermo.garciasaez@uclm.es} \\
			[10pt] 
		\end{tabular}\\
		{\centering{\Large José Camilo Rueda-Niño}{\small\textsuperscript{2}} \\
			\mail{jcrueda@fing.edu.uy}} \\ \vspace{5mm}
		
		\textsc{\textsuperscript{1}ETSI Industrial\\ Departamento de Matem\'aticas, Universidad de Castilla-La Mancha} \\
		Campus Universitario s/n, 13071-Ciudad Real, SPAIN. \\ \vspace{5mm}
		
		\textsc{\textsuperscript{2}PEDECIBA – Program for the Development of Basic Sciences and
			Instituto de Matemática y Estadística “Rafael Laguardia”, Universidad de la Republica,
			Montevideo, Uruguay}
		\vspace{5mm}
	\end{center}
	\begin{Abstract}

    We study the nonlocal $\rho$-Laplacian, defined as the composition of the nonlocal
divergence and gradient operators associated with a general radial kernel $\rho$: $\Delta_\rho u=\mbox{div}_\rho\left(D_\rho u\right)$.
Our first main contribution is to establish a precise connection between this operator
and the class of integro-differential elliptic operators studied by Fern\'andez-Real
and Ros-Oton (\cite{FernandezRos}), identifying explicit conditions on the kernel $\rho$ that guarantee
membership in this class. Our second main contribution concerns maximum and comparison principles for the $\rho$-Laplacian. We establish both a strong
and a weak maximum principle under conditions on $\rho$ that are strictly weaker
than those required for membership in the integro-differential class, thereby
covering a genuinely broader family of operators. The results require only minimal
assumptions on the kernel, and in particular do not rely on any fractional-type
comparability condition.
	\end{Abstract}
	
	\noindent {\bf }

    \noindent {\bf Keywords:} Nonlocal gradients, nonlocal Laplacians, integro-differential equations, maximum principles.

	\noindent {\bf AMS Subject Classification:} 35B50, 35B51, 35K09, 47G20.
	
	\tableofcontents
	
	\section{Introduction}

Many phenomena arising in nature, science and engineering are traditionally modeled through
differential equations or local variational principles. The underlying assumption of locality
means that the behavior of a system at a given point is determined solely by its immediate
neighborhood. However, there are numerous situations in which long-range interactions play a
fundamental role and cannot be adequately captured by purely local descriptions. This has
motivated the development of nonlocal models, which are formulated in terms of integral or
integro-differential equations and operators. Such models offer several advantages over their
classical counterparts: they are capable of encoding multiscale behavior and long-range
interactions, they typically impose weaker regularity requirements on the admissible functions,
and they provide a natural framework for describing discontinuities and singularities that would
otherwise be difficult to handle within a local theory. As a result, nonlocal models have found
application across a wide range of disciplines, including materials science, continuum
mechanics, diffusion processes, image processing and machine learning. A thorough account
of nonlocal modeling is given in~\cite{Du}.

A particularly relevant class of nonlocal operators consists of integro-differential operators,
which generalize classical differential operators by replacing pointwise derivatives with integral
expressions involving differences of the function over a domain. Typically, these operators adopt
the form
\begin{equation}\label{eq:integrodifferential}
	\mathcal{L}(u)(x)=\int_{\mathbb{R}^n} (u(x)-u(x+h))K(h)\,dh,
\end{equation}
where $K$ is a positive kernel. A systematic and modern treatment of such operators can be
found in the monograph~\cite{FernandezRos}, where the theory is developed in considerable
depth and generality. The regularity theory for both linear and nonlinear integro-differential
operators has also been extensively investigated in
\cite{Kassmann2009,KassmannRangSchwab2014}. Among the operators studied in this framework,
the fractional Laplacian $(-\Delta)^s$, $s\in(0,1)$, occupies a central place as a canonical
example of this class. It constitutes a natural generalization of the classical Laplacian to
non-integer orders of differentiability and arises in a variety of contexts, including anomalous
diffusion, potential theory, and the study of L\'evy processes. Its analytical properties,
including mapping properties between function spaces, maximum principles, and regularity theory,
have been extensively investigated in recent years. Foundational contributions to the theory
include \cite{CaSi07,DeGu13,RoSe14,Ros-Oton}; we refer to
\cite{Abantangelo,TenDefis} for more recent developments. On the numerical side, finite element
methods and other discretization techniques for fractional diffusion problems have been
developed and analyzed in \cite{AcostaBorthagaray2017,Bonito2018, Borthagaray2025}.

Alongside operators of diffusion type such as the fractional Laplacian, there has been growing
interest in nonlocal operators of gradient type, which provide a framework for a nonlocal vector
calculus~\cite{DeGuOlKa21, DGLZ, MeS, Silhavy2019}. These operators are typically defined
in terms of a kernel $\rho$ that we assume to be radial and singular at the origin; for a
function $u \colon \Omega \to \mathbb{R}$, the nonlocal gradient of $u$ with kernel $\rho$ is
defined as
\begin{equation} \label{eq:nonlocal_gradient}
D_\rho u(x) = \int_\Omega \frac{u(x) - u(y)}{|x - y|} \frac{x - y}{|x - y|}
\,\rho(x - y)\, dy.
\end{equation}
One of the distinguishing features of these nonlocal gradient operators is that, like the
fractional Laplacian, they require less regularity from the functions on which they act than
their classical counterparts, while still retaining a rich enough structure to support a
meaningful calculus and variational theory.

A particularly well-studied instance within this class is the Riesz $s$-fractional gradient,
$D^s u$, which has attracted considerable attention following the foundational works of Shieh
and Spector~\cite{ShiehSpector2015, ShiehSpector2018}. It corresponds to the kernel
$\rho = c_{n,s}|x|^{-(n+s-1)}$, $s\in (0,1)$, where $c_{n,s}$ is a suitable normalization
constant. A key feature of the Riesz fractional gradient is that the associated functional
spaces are equivalent to Bessel potential spaces $H^{s,p}(\mathbb{R}^n)$, providing a
fine functional-analytic framework for applications and variational models; see
\cite{BellidoCuetoGarcia2025, BellidoGarciaI2025}. Applications of the Riesz fractional
gradient include \cite{Almi2025, Almi2025II, BeCuMC, BeCuMC21, Borthagaray2025, Camilo2026, camposI, camposII, camposIII, CamposGarciaSaez2026, Carrero2025, COMI2019, Gambera2024, Garcia2025, Hidde2022, Silhavy2024}.

A further example is the truncated Riesz fractional gradient $D^s_\delta u$, obtained by
taking the kernel $\rho = \omega_\delta(x)|x|^{-(n+s-1)}$, where $\omega_\delta$ is a
cut-off function supported on the ball of radius $\delta > 0$, the
\emph{horizon} of nonlocal interaction, centered at the origin. This variant is motivated by applications requiring
work on bounded domains. Unlike the Riesz fractional gradient, the natural functional space
associated with $D^s_\delta$ is not a classical one, and establishing the minimal structural
properties of this space---Poincar\'{e}-type inequalities and compact embeddings---required
for variational analysis has been one of the central tasks in the early development of the
theory. This was initiated in~\cite{BellidoCuetoMora2023}, and since then further work and
applications in solid mechanics have been developed in~\cite{BellidoCuetoMoraII2023, Bellido_NH2024, BellidoCueto2024, BellidoGarcia2026, Cueto2023, Cueto2026, Kreisbeck2024}. 

Finally, for a general radial kernel $\rho$, the functional spaces associated with $D_\rho$
have been studied in~\cite{BellidoMoraHidde2025}, where optimal conditions on $\rho$
guaranteeing Poincar\'{e} inequalities and compact embeddings are identified. See
also~\cite{BellidoGarcia2026, CuetoHidde2025, Hidde2026}. 

In this paper we deepen the study of nonlocal gradients for general radial kernels $\rho$
through the analysis of the nonlocal $\rho$-Laplacian $\Delta_\rho$, defined---in analogy
with the local case---as the composition of the nonlocal $\rho$-divergence and the nonlocal
gradient $D_\rho$. We investigate the connection between nonlocal gradients and
integro-differential operators, showing that $\Delta_\rho$ can be written as an operator of
the form~\eqref{eq:integrodifferential}. We also determine conditions on the kernel $\rho$
under which $\Delta_\rho$ belongs to the class $\mathcal{K}_s(\lambda,\Lambda)$ considered
in~\cite{FernandezRos}. We regard this as a significant connection between the two theories,
which can be exploited in either direction: on the one hand, the extensive theory of
integro-differential operators applies to variational models based on nonlocal gradients; on
the other hand, the variational structure of nonlocal operators, grounded in an integration
by parts formula, gives rise to natural nonlocal differential objects that parallel their
classical counterparts. We also study maximum and comparison principles for the
$\rho$-Laplacian. Our results require only the minimal hypothesis on $\rho$ (hypothesis
\eqref{eqn:H0} below), so that $\Delta_\rho$ need not belong to the class $\mathcal{K}_s(\lambda,\Lambda)$,
thereby generalizing the results in~\cite{FernandezRos} to a strictly larger class of operators.

The outline of the paper is the following.
Section 2 fixes the notation used throughout.
Section 3 collects the necessary background material:
we recall the nonlocal gradient framework and the associated function spaces $H^{\rho,p}$,
the scaled kernel construction, and the class of integro-differential operators
$\mathcal{K}_s(\lambda,\Lambda)$ from~\cite{FernandezRos}.
Section 4 introduces the nonlocal $\rho$-Laplacian $\Delta_\rho$,
establishes its representation as a convolution operator,
studies the associated nonlocal Gagliardo spaces $W^{\rho,2}$ and their equivalence
with $H^{\rho,2}$, and analyzes the relationship between $\Delta_\rho$ and the
class $\mathcal{K}_s(\lambda,\Lambda)$.
Finally, Section 5 is devoted to maximum and comparison principles:
we prove both a strong maximum principle, valid under the sole assumption~(H0) on $\rho$,
and a weak maximum principle in the variational setting of $H^{\rho,2}_0(\Omega)$.

\section{Notation}
	We fix $n\in \mathbb{N}$ the dimension of our ambient space $\R^n$ and we will denote by $\Omega\subset \R^n$ and open bounded subset representing the body. The notation for Sobolev $W^{1,p}$ and Lebesgue $L^p$ spaces is the standard one, as is that of smooth
	functions of compact support $C_c^\infty$. We will
	indicate the domain of the functions, as in $L^p
	(\Omega$); the target is indicated only if it is not $\R$. The notation for the classical gradient on its weak sense is the usual one $D$. For a Banach space $X$, we denote its topological dual as $X^*$, and the usual {\it strong} convergence in normed spaces is denoted as $\to$.

	We write $|x|=\left(\sum_{j=1}^n x_j^2\right)^{1/2}$ for the Euclidean norm of a vector $x=(x_1,\ldots,x_n)\in \R^n$. The ball centered at $x\in \R^n$ and with radius $r>0$ is denoted by $B(x,r)=\{y\in \R^n: |x-y|<r\}$ or if there is no ambiguity on the center, just by $B_r$. The complementary of a set $E\subset \R^n$ is denoted by $E^c=\R^n\setminus{E}$, its closure by $\overline{E}$ and its boundary by $\partial E$.
	
	Our convention for the Fourier transform of functions $f\in L^1(\R^n)$ is $$\widehat{f}(\xi)=\int_{\R^n}f(x)e^{-2\pi ix\cdot \xi}\,dx,\,\xi\in \R^n.$$ This definition is extended by continuity and duality to other function and distribution spaces as usually in function spaces theory. We will sometimes use the alternative notation $\mathfrak{F}(f)$ for $\widehat{f}$. More details of this operator could be found in the classical text \cite{Grafakos2008}. We denote by $\mathcal{S}$ the Schwartz space and $\mathcal{S}'$ the space of tempered distributions.
	
	Regarding radial functions we have the following definitions:
	\begin{itemize}
		\item A function $f:\R^n\to \R$ is \textit{radial} if there exists $\overline{f}:[0,\infty)\to \R$ such that $\overline{f}(|x|)=f(x)$ for every $x\in \R^n$. The function $\overline{f}$ is called the \textit{radial representation} of $f$.
		\item A radial function $f:\R^n\to \R$ is \textit{radially decreasing} if its radial representation is a
		decreasing function.
		\item A function $g:\R^n\to \R^n$ is \textit{vector radial} if there exists a radial function $\overline{\phi}:[0,\infty)\to \R$  such that
		$\phi(x)=\overline{\phi}(|x|)x$ for every $x\in \R^n$.
	\end{itemize}
	
	For real valued functions, we use the monotonicity properties of being increasing and decreasing in
	the non-strict sense. A function $f :\R \to\R$ is called \textit{almost decreasing} if there exists a positive constant $C$ such
	that $f(t)\geq Cf(s)$ for every $t\leq s$, and an analogous definition for \textit{almost increasing}.
\section{Preliminaries}\label{sec:preliminaries}

This section collects the foundational material required throughout the paper.
We begin by recalling the nonlocal gradient framework associated with a kernel
\( \rho \), together with the variational function spaces arising from it, which will serve as the analytical backbone for the subsequent treatment. This
functional setting will provide the natural framework for the nonlocal elliptic operator driven by \( \rho \) that we introduce in the following sections. The presentation follows mainly \cite{BellidoCuetoMora2023, BellidoMoraHidde2025, CuetoHidde2025}, to which we refer the reader for further details and proofs. Finally, we recall some basic properties on the class of integro-differential operators introduced by Fernández-Real and Ros-Oton in \cite{FernandezRos}, which are closely related to the ones we consider in this work.
\subsection{Nonlocal gradients}
We assume throughout that the kernel \( \rho \) satisfies \eqref{eqn:H0},
which provides the natural and necessary conditions to establish a suitable
functional framework; see \cite{BellidoMoraHidde2025} for a detailed discussion.
\begin{equation}\label{eqn:H0} \tag{H0}
	\begin{cases}
		\text{\( \rho:\mathbb{R}^n\setminus\{0\} \to [0,\infty) \) is radial}, \\
		\text{\( \rho \in L^1_{\text{loc}}(\mathbb{R}^n) \) with 
			\( \displaystyle \int_{\mathbb{R}^n} \min\{1,|x|^{-1}\}\rho(x)\,dx < \infty \),} \\
		\text{\( \inf_{\overline{B(0,\varepsilon)}}\rho >0 \) for some \( \varepsilon>0 \).}
	\end{cases}
\end{equation}
We now turn to the definition of the nonlocal gradient.
	\begin{Defi}[Nonlocal gradient]\em
		Let $u\in C_c^\infty(\R^n)$, we define the \textit{nonlocal gradient} with kernel $\rho$ of $u$ as $$D_{\rho}u(x):=\int_{\R^n}\frac{u(x)-u(y)}{|x-y|}\frac{x-y}{|x-y|}\rho(x-y)\,dy,\,x\in \R^n.$$
	\end{Defi}
	
	It is straightforward to see that for functions $u\in C_c^\infty(\R^n)$, the integral involving the definition of $D_\rho u$ is absolutely convergent for each $x\in \R^n$. Moreover, $D_\rho u\in L^1(\R^n;\R^n)\cap L^\infty(\R^n;\R^n)$ (\cite{BellidoMoraHidde2025}), and hence $D_\rho u\in L^p(\R^n;\R^n)$ for each $1<p<\infty$ by Riesz-Torin interpolation inequality.\\
	
	In \cite{BellidoMoraHidde2025, CuetoHidde2025}, examples of admissible kernels are given, the most paradigmatic one 
	being $$\rho^s(x):=c_{n,s}\frac{1}{|x|^{n+s-1}},\,s\in (0,1),$$ where $c_{n,s}$ is a suitable normalization constant. The nonlocal gradient associated with this kernel is the well-known Riesz fractional gradient $D^s=:D_{\rho^s}$, which is defined for functions $u\in C_c^\infty(\R^n)$ as $$D^su(x)=c_{n,s}\int_{\R^n}\frac{u(x)-u(y)}{|x-y|^{n+s}}\frac{x-y}{|x-y|}\,dy=D(I_{1-s}u),$$ where $I_s$ is the Riesz potential.
	The Riesz fractional gradient was introduced by Shieh and Spector in \cite{ShiehSpector2015,ShiehSpector2018}, and since then has been widely studied.
	
	Another important example concerns a truncated version of the Riesz fractional gradient introduced in \cite{BellidoCuetoMora2023}, which is suitable for bounded domains. This formulation is motivated by concepts from peridynamics, as its integration domain is restricted to a ball of radius $\delta>0$, corresponding to the interaction horizon among particles in the peridynamic framework. Let $\delta>0$, and a non-negative radial cut-off
	function $w_\delta$ such that $w_\delta(0)>0$, $w_\delta\in C_c^\infty\left(B(0,\delta)\right)$ and $w_\delta/|\cdot|^{1+s}$ is radially decreasing.
	Now, taking the kernel 
	$$\rho_\delta^s(x):=\frac{w_\delta(x)}{|x|^{n+s-1}},$$ 
	yields that $D_{\rho_\delta^s}$ is, up to a normalization constant, equal to $D_\delta^s$, the nonlocal gradient introduced in \cite{BellidoCuetoMora2023}.
	
	The theory established for $D^s$ and $D^s_\delta$ has been further generalized to encompass nonlocal gradients $D_\rho$, where the kernel $\rho$ satisfies hypothesis \eqref{eqn:H0}, as shown in \cite{BellidoMoraHidde2025}. The present work is developed within this broader framework, and in this section we provide an overview of the fundamental aspects of the calculus involved, together with the underlying functional spaces.
	
	The first important result is the following representation formula for $D_\rho$ in terms of a convolution with the locally integrable function 
	$$Q_\rho(x):=\int_{|x|}^\infty \frac{\overline{\rho}(t)}{t}\,dt,\qquad D Q_\rho(x)=-\frac{\overline{\rho}(|x|)}{|x|}\frac{x}{|x|},\,x\in \R^n\setminus\{0\},$$ where $\overline{\rho}$ is the radial representation of the kernel $\rho$.
	By \cite[Proposition 2.6]{BellidoMoraHidde2025}, we have that $$D_\rho u=Q_\rho*D u=D\left(Q_\rho*u\right),$$ 
	for $u\in C_c^\infty(\R^n)$, and moreover, if $\rho\in L^1(\R^n)$, then $Q_\rho\in L^1(\R^n)$ (in fact, $\rho$ having compact support is enough for $Q_\rho$ to lie in $L^1(\R^n)$, \cite[Lemma 2.5 (iii)]{BellidoMoraHidde2025}). Also, taking Fourier transform, 
	
	$$\widehat{D_\rho u}(\xi)=2\pi i\xi \widehat{Q_\rho}(\xi)\widehat{u}(\xi):=\lambda_\rho(\xi)\widehat{u}(\xi),$$ with $$\widehat{Q_\rho}(\xi)=\frac{1}{2\pi|\xi|}\int_{\R^n}\frac{\rho(x)x_1}{|x|^2}\sin(2\pi|\xi|x_1)\,dx,\,\xi\in \R^n\setminus{\{0\}}.$$
	
	We can also define the nonlocal divergence for general kernels $\rho$.
	\begin{Defi}\em
		For $v\in C_c^\infty(\R^n;\R^n)$, we define the {\it nonlocal divergence} with kernel $\rho$ as 
		$$\operatorname{div}_\rho v(x):=\int_{\R^n}\frac{v(x)-v(y)}{|x-y|}\cdot \frac{x-y}{|x-y|}\rho(x-y)\,dy=(D_\rho \cdot v)(x),\,x\in \R^n.$$ 
	\end{Defi}
	As in the local setting, we also have a duality relationship between $D_\rho$ and $\operatorname{div}_\rho$ by \cite[Proposition 3.2]{BellidoMoraHidde2025}.
	\begin{Lema}[Integration by parts]\label{IBP}
		Let $u\in C_c^\infty(\R^n)$ and $v\in C_c^\infty(\R^n;\R^n)$. Then, $$\int_{\R^n}D_\rho u\cdot v\,dx=-\int_{\R^n}u\operatorname{div}_\rho v\,dx.$$
	\end{Lema}
	The representation formula for $D_\rho$ could be also extended for the nonlocal divergence. In particular, since $\operatorname{tr}D_\rho u=\operatorname{div}_\rho u$, we have that $$Q_\rho*\operatorname{div}v=Q_\rho*(\operatorname{tr}D v)=\operatorname{tr}(Q_\rho*v)=\operatorname{tr}D_\rho v=\operatorname{div}_\rho v.$$ On the other hand, $$\operatorname{div}_\rho v=\operatorname{tr}D_\rho v=\operatorname{tr}D (Q_\rho*v)=\operatorname{div}(Q_\rho*v).$$
	\subsection{Nonlocal Sobolev spaces}
	
	In analogy to the fractional Sobolev spaces $H^{s,p}(\R^n)$, commonly known in the literature as Bessel potential spaces, for $1<p<\infty$, the $\rho$-nonlocal Sobolev spaces $H^{\rho,p}(\R^n)$ were introduced in \cite{BellidoMoraHidde2025} as $$H^{\rho,p}(\R^n):=\{u\in L^p(\R^n): D_\rho u\in L^p(\R^n;\R^n)\},$$ endowed with the norm $$\norm{u}_{\rho,p}:=\norm{u}_p+\norm{D_\rho u}_p.$$ As a consequence of \cite[Theorem 3.9 (i)]{BellidoMoraHidde2025} and the completeness of those spaces, we can alternatively define them as
	\begin{equation}\label{def-non-local-soblev}
		H^{\rho,p}(\R^n):=\overline{C_c^\infty(\R^n)}^{\norm{\cdot}_{\rho,p}}.
	\end{equation}

    Clearly, by density, definition of $D_\rho$ can be extended to $H^{\rho,p}(\R^n)$. For the choice $\rho=\rho^s, s\in (0,1)$, we have that by \cite[Theorem 1.7]{ShiehSpector2015}, $$H^{\rho,p}(\R^n)=H^{s,p}(\R^n),$$ where $H^{s,p}(\R^n)$ is the classical Bessel potential space (see \cite{BellidoCuetoGarcia2025,BellidoGarciaI2025} and the references therein). 
	For an open set $\Omega\subset\R^n$, we define the closed subspace $$H^{\rho,p}_0(\Omega):=\left\{u\in H^{\rho,p}(\R^n): u(x)=0\text{ a.e. }x\in\Omega^c\right\}.$$ 
	If $\Omega$ is a Lipschitz domain, by \cite[Theorem 3.9 (iii)]{BellidoMoraHidde2025} we have that $$H^{\rho,p}_0(\Omega)=\overline{C_c^\infty(\Omega)}^{\norm{\cdot}_{\rho,p}},
	$$ where the elements $C_c^\infty(\Omega)$ should be interpreted as its extension to $\R^n$ by zero. Those spaces are complete, reflexive and separable for every $1<p<\infty$. For $p=2$, the spaces $H^{\rho,2}(\R^n)$ and  $H^{\rho,2}_0(\Omega)$ are Hilbert spaces. 
	In order to ensure the fundamental structural properties of the spaces $H^{\rho,p}$, such as Poincar\'e-type inequalities and compact embeddings, it is necessary to impose additional assumptions on the kernels $\rho$. Let $\varepsilon$ be as in \eqref{eqn:H0}. The main hypothesis we will employ are:
	
	\begin{enumerate}[label=(H\arabic*),ref=H\arabic*]
		\item\label{eqn:H1} There exists $\nu>0$ such that the functions $f_\rho:(0,\infty)\to \R$, defined by $f_\rho(r)= r^{n-2}{\rho}(r)$, and $g(r)= r^\nu f_\rho(r)$ are decreasing on $(0,\varepsilon)$;
		\item\label{eqn:H2} The function $f_\rho$ is smooth away from the origin, and for every positive integer $k$ there exists a constant $C(k)>0$ such that		
		\[
		\left|\frac{d^kf_\rho(r)}{dr^{k}}\right|\leq C(k)\,r^{-k}f_\rho(r), \quad r\in (0,\varepsilon);
		\]
        \item\label{eqn:H3} There exists $s\in(0,1)$ such that the function $r\mapsto r^{n+s-1}{\rho}(r)$ is almost decreasing on $(0,\varepsilon)$;
		\item\label{eqn:H4} There exists $t\in(0,1)$ such that the function $r\mapsto r^{n+t-1}{\rho}(r)$ is almost increasing on $(0,\varepsilon)$.
	\end{enumerate}
    Furthermore, we could restrict ourselves to kernels with compact support without loss of generality since the main properties of the function spaces $H^{\rho,p}(\R^n)$ and $H_0^{\rho,p}(\Omega)$ are encoded on the behavior of $\rho$ near
	zero by \cite[Proposition 3.10]{BellidoMoraHidde2025}. 
    Now, Poincaré inequality and the compactness results reads as follows (see \cite[Theorem 4.11]{BellidoMoraHidde2025}): 
	\begin{Teor}\label{Poincare-ineq}
		Let $\rho$ a kernel satisfying (\ref{eqn:H0}) with compact support and $1<p<\infty$. Furthermore, if $p=2$ we assume that $\rho$ satisfies (\ref{eqn:H1}) and if $p\not =2$ we assume that $\rho$ satisfies (\ref{eqn:H1}) and (\ref{eqn:H2}). Then, the following two statements hold: 
        \begin{enumerate}
            \item If $\lim_{t\to 0^+}t^{n-1}\overline{\rho}(t)>0$, then there exists a constant $C=C(\Omega,n,p,\rho)>0$, such that $$\norm{u}_{L^p(\Omega)}\leq C\norm{D_\rho u}_{L^p(\R^n;\R^n)},\,u\in H_0^{\rho,p}(\Omega).$$
            \item If If $\lim_{t\to 0^+}t^{n-1}\overline{\rho}(t)=\infty$, then $H^{\rho,p}_0(\Omega)$ is compactly embedded into $L^p(\Omega)$.
        \end{enumerate}
	\end{Teor}
	These hypotheses are quite natural, and we refer to Sections~4 and~5 of \cite{BellidoMoraHidde2025} for a detailed discussion. As expected, they are satisfied by the kernels associated with $D^s$ and $D_\delta^s$. Hypothesis on Theorem \ref{Poincare-ineq} are almost optimal (see \cite[Proposition 7.5]{BellidoMoraHidde2025}). 
    
    The further hypothesis (\ref{eqn:H3}) and (\ref{eqn:H4}) are related to the behavior near zero of the kernel in comparison with purely fractional kernels, which is also a quite natural thing to recall.
    These extra hypothesis are required in order to prove some sort of nonlocal fundamental theorem of calculus \cite[Theorem 5.2]{BellidoMoraHidde2025}, i.e., the existence of a vector radial function $V_\rho\in C^\infty(\R^n\setminus\{0\};\R^n)\cap L^1_{loc}(\R^n;\R^n)$ such that for every $u\in H^{\rho,p}_0(\Omega)$, $$u(x)=\int_{\R^n}V_{\rho}(x-y)\cdot D_\rho u(y)\,dy,\,x\in\R^n.$$
    
    Furtheromre, assuming (\ref{eqn:H0})--(\ref{eqn:H4}) for $\rho$, we obtain the following lemma, which establishes the existence of some translation operators from the nonlocal framework to the local one and vice versa \cite[Lemma~2.12]{CuetoHidde2025}.
	\begin{Lema}[Translation operators]\label{Translation}
		Let $\rho$ satisfy (\ref{eqn:H0})--(\ref{eqn:H4}) with compact support. The linear map $\mathcal{Q}_\rho$ defined as $\mathcal{Q}_\rho u:=Q_\rho *u$, $u\in C_c^\infty(\R^n)$, extends to a linear bounded map from $H^{\rho,p}(\R^n)\to W^{1,p}(\R^n)$. Furthermore, it holds for all $u\in H^{\rho,p}(\R^n)$ that $$D_\rho u=D\left(\mathcal{Q}_\rho u\right).$$ The linear map $\mathcal{P}_\rho$ defined as $$v\mapsto\mathcal{P}_\rho  v:=\mathfrak{F}^{-1}\left(\widehat{v}/\widehat{Q}_\rho\right),\,v\in \mathcal{S}(\R^n),$$ extends to a bounded operator from $W^{1,p}(\R^n)$ to $H^{\rho,p}(\R^n)$. Moreover, $(\mathcal{P}_\rho)^{-1}=\mathcal{Q}_\rho$, and for every $v\in W^{1,p}(\R^n)$, $$D v=D_\rho(\mathcal{P}_\rho v).$$
	\end{Lema}
    Nonlocal Sobolev embeddings have been established as well in \cite[Theorem 6.2]{BellidoMoraHidde2025}.
	\begin{Teor}\label{Embedding}
		Let $\rho$ satisfy (\ref{eqn:H1})--(\ref{eqn:H4}) with compact support and pick $s,t$ as in those hypothesis. Let $1<p<\infty$ such that $tp<n$. Then, $$H^{\rho,p}_0(\Omega)\xhookrightarrow{}L^q(\Omega),\,1\leq q\leq p_s^*,$$ where $p_s^*=\frac{np}{n-sp}$.
	\end{Teor}
	Moreover, when $sp>n$, \cite[Theorem 6.5]{BellidoMoraHidde2025}, yields that the embedding extends to the case $q=\infty$.
	\subsection{Scaled Kernels}
	In \cite{CuetoHidde2025}, a scaled version of the $\rho$ kernels is introduced in order to study $\Gamma$-convergence results for general nonlocal gradients with varying horizons. Starting from a kernel $\rho$ under the hypothesis (\ref{eqn:H0})--(\ref{eqn:H4}) and normalized in the sense that $$\operatorname{supp}\rho=\overline{B(0,1)},\,\norm{\rho}_1=n,$$
	they define the rescaled family of kernels $(\rho_\delta)_\delta$ for horizons $\delta>0$ as $$\rho_\delta(x):=c_\delta\rho(x/\delta),\,x\in \R^n\setminus\{0\}.$$ Here $(c_\delta)_\delta\subset (0,\infty)$ is a suitable sequence of scaling factors. In particular, they are chosen as $c_\delta=\delta^{-n}$ when the limit $\delta\to 0^+$ is studied, and as $c_\delta=\overline{\rho}(1/\delta)^{-1}$ for $\delta\to \infty$.
	This rescaling preserve the main properties of $\rho$, as they verify the hypothesis \eqref{eqn:H0}--\eqref{eqn:H4}. Moreover, $\operatorname{supp}\rho_\delta=\overline{B(0,\delta)}$, which makes the associated gradiente $D_{\rho_\delta}$ a nonlocal gradient of finite horizon $\delta>0$. The kernel $Q_{\rho_\delta}$ associated to $\rho_\delta$, satisfies that 
    $$Q_{\rho_\delta}(x)=c_\delta Q_\rho(x/\delta),\qquad \widehat{Q_{\rho_\delta}}(\xi)=c_\delta\delta^{-n}\widehat{Q_{\rho}}(\delta\xi).$$ For the choice $c_\delta=\delta^{-n}$, it is proven in \cite[Theorem 3.1]{CuetoHidde2025} the following localization result:
	
	\begin{Lema}[Localization for vanishing horizon]\label{Loc1}
		Let $\rho_\delta$ a rescaled kernel with $c_\delta=\delta^{-n}$. Then, the following hold:
		\begin{itemize}
			\item For each $\varphi\in C_c^\infty(\R^n)$, and all $\delta\in (0,1]$, $$\norm{D_{\rho_\delta}\varphi-D\varphi}_\infty\leq \delta^2\operatorname{
				Lip}(D^2\varphi).$$ In particular, $D_{\rho_\delta}\varphi\to D\varphi$ uniformly on $\R^n$ as $\delta\to 0^+$.
			\item For each $u\in W^{1,p}(\R^n)$, $u\in H^{\rho_\delta,p}(\R^n)$ for all $\delta\in (0,1]$, and $D_{\rho_\delta}u\to D u$ in $L^p(\R^n;\R^n)$ as $\delta\to 0^+$.
		\end{itemize}
	\end{Lema}
This result shows the expected behavior of the operator depending on the choice of \( \delta \). As the nonlocality vanishes, one recovers the corresponding local operators. On the other hand, when \( \delta \to \infty \), the nonlocal interactions extend to the whole space and, under assumptions \eqref{eqn:H3}--\eqref{eqn:H4}, one naturally recovers a purely fractional model associated with the Riesz fractional gradient, see \cite[Lemma 4.2]{CuetoHidde2025}, and \cite[Proposition 4.6]{CuetoHidde2025}.

\subsection{Integro-Differential operators}
In \cite{FernandezRos}, the following class of integro-differential operators $\mathcal{L}$ is extensively and deeply studied: Let $\mathcal{K}_s(\lambda,\Lambda)$ as the class of integro-differential operators $\mathcal{L}$ where $$\mathcal{L}u(x):=\operatorname{pv}_x\int_{\R^n}\left((u(x+y)-u(x)\right)K(y)\,dy,$$ such that the kernel $K$ is nonnegative and symmetric ($K(y)=K(-y)$ for any $y$), $K\in L^1_{loc}(\R^n\setminus\{0\})$, the following integrability conditions holds: \begin{equation}\label{eq:integrabilityK}
\int_{\R^n} \min\{1,|z|^2\}K(z)\,dz<\infty,
\end{equation}
and its Fourier symbol, which can be computed as $$m_K(\xi)=\int_{\R^n}(1-\cos(2\pi \xi\cdot z)K(z)\,dz,$$ must be comparable to the symbol of the fractional Laplacian, i.e., there must exists two positive constants $\lambda,\Lambda$ such that 
\begin{equation}\label{eq:fourierK}
0<\lambda |\xi|^{2s}\leq m_K(\xi)\leq \Lambda|\xi|^{2s},\,\xi\in \R^n.
\end{equation}
	In \cite[Proposition 2.2.1]{FernandezRos}, it is proven that the condition on the Fourier symbol of $K$ is equivalent to the existence of two positive constants $c_1,c_2$ such that \begin{align*}
		r^{2s}\int_{B_{2r}\setminus B_r}K(y)\,dz\leq c_2,\\
		c_1\leq r^{2s-2}\inf_{e\in \mathbb{S}^{n-1}}\int_{B_r}|e\cdot y|^2 K(y)\,dy.
	\end{align*}
	Observe that this class of general integro-differential elliptic operators generalizes the fractional Laplacian. Indeed, if we choose $K(y)=|y|^{-(n+2s)}$, we have that up to a constant, $\mathcal{L}_K=(-\Delta)^s$. Clearly, the condition under the Fourier multiplier is verified for any $0<\lambda\leq 1\leq \Lambda$, and the integrability condition follows easily as well since   $$\int_{B_1}|y|^2K(y)\,dy=c_{n,s}\int_{B_1}|y|^{2-n-2s}\,dy=c_{n,s}\omega_{n-1}\int_0^1 r^{2-n-2s}r^{n-1}=c_{n,s}w_{n-1}\int_0^1 r^{1-2s}\,dr,$$ which is finite if and only if $1-2s>-1$, i.e., $s<1$. Analogously, $$\int_{B_1^c}|y|^{-n-2s}\,dy=\omega_{n-1}\int_1^\infty r^{-2s-1}\,dr,$$ which is finite if and only if $0<s<1$. 
	
	This class of operators is deeply rooted on probability theory and its interplay with the modern theory of partial differential equations. We refer to \cite{FernandezRos,Foghem2025} and the references therein in this direction. 
    
	\section{The nonlocal $\rho$-Laplacian} \label{non-rho-lap}
Recall that the classical Laplacian is defined for sufficiently smooth functions
\( u \) as \( \Delta u = \operatorname{div}(D u) \). In the fractional
setting, it was shown in \cite{ShiehSpector2015} that the analogous identity
\( -\operatorname{div}_s(D^s u) = (-\Delta)^su \) holds for the Riesz
fractional gradient. It is therefore natural to seek a generalization of this
notion to general kernels \( \rho \). Indeed, in view of the definitions of
the nonlocal gradient and divergence introduced above and their resemblance
with their local counterparts
\begin{Defi} Given $u \in C_c^\infty(\mathbb{R}^n)$, we define the nonlocal Laplacian of kernel $\rho$, or simply the $\rho$-Laplacian, of $u$ as 
    \[
\Delta_\rho u := \operatorname{div}_\rho(D_\rho u), \quad u \in C_c^\infty(\mathbb{R}^n).
\]
\end{Defi}

This operator has already appear in a natural manner in the study of nonlocal models such as the general nonlocal linear elasticity model from \cite{BellidoGarcia2026}. 

Our aim now is to recast this operator  in a form analogous to the one familiar from the fractional Laplacian. To this end, we provide an alternative characterization of the \( \rho \)-Laplacian as a Fourier multiplier with symbol \( -4\pi^2|\xi|^2\widehat{Q_\rho}(\xi)^2 \). Actually, this establishes the remarkable fact that the $\rho$-Laplacian is a convolution operator, as is that of $D^\rho$.

\begin{Prop}
  For any $u\in C_c^\infty(\mathbb{R}^n)$, it holds
  \[
\Delta_\rho u = \int_{\mathbb{R}^n} K_\rho(x-y)\,u(y)\,dy,
\]
where the convolution kernel \( K_\rho \) is given by
\[K_{\rho}(x):= -D Q_\rho * D Q_\rho(x) = -\int_{\mathbb{R}^n}
\overline{\rho}(|z|)\,\overline{\rho}(|x-z|)\,
\frac{z}{|z|} \cdot \frac{x-z}{|x-z|}\,dz.
\]
Moreover, 
 \[\widehat{K_\rho}(\xi) = -4\pi^2|\xi|^2 \left(\widehat{Q}_\rho\right)^2(\xi)=-|\lambda_\rho(\xi)|^2=:m_{K_\rho}(\xi),\]
\end{Prop}
\begin{proof}
    Since
\( D_\rho u = D(Q_\rho * u) \) and
\( \operatorname{div}_\rho v = \operatorname{div}(Q_\rho * v) \) for
\( u \in C_c^\infty(\mathbb{R}^n) \) and \( v \in C_c^\infty(\mathbb{R}^n;\mathbb{R}^n) \),
a direct computation gives
\begin{align*}
	\widehat{\Delta_\rho u}(\xi)
	&= \widehat{\operatorname{div}_\rho(D_\rho u)}(\xi)
	= 2\pi i\,\widehat{Q}_\rho(\xi)\,\xi \cdot \widehat{D_\rho u}(\xi)
	= 2\pi i\,\xi \cdot \widehat{Q}_\rho(\xi)
	\left(2\pi i\,\xi\,\widehat{Q}_\rho(\xi)\widehat{u}(\xi)\right) \\
	&= -4\pi^2|\xi|^2\widehat{Q}_\rho(\xi)^2\widehat{u}(\xi).
\end{align*}
Equivalently, we may write
\[
\widehat{\Delta_\rho u}(\xi) = \widehat{K_\rho}(\xi)\,\widehat{u}(\xi),
\]
with
\[
K_\rho(x) := -D Q_\rho * D Q_\rho(x),
\]
and therefore
\[\widehat{K_\rho}(\xi) = -4\pi^2|\xi|^2\left(\widehat{Q}_\rho\right)^2(\xi),
\]
so that \( \Delta_\rho u = K_\rho * u \), or explicitly,
\[
\Delta_\rho u = \int_{\mathbb{R}^n} K_\rho(x-y)\,u(y)\,dy,
\qquad
K_{\rho}(x) = -\int_{\mathbb{R}^n}
\overline{\rho}(|z|)\,\overline{\rho}(|x-z|)\,
\frac{z}{|z|} \cdot \frac{x-z}{|x-z|}\,dz.
\]
\end{proof}
In the sequel, we will use the notation for the symbol $m_{K_\rho}(\xi):=\widehat{K_\rho}(\xi)$. 

We further note that the operator \( K_\rho \) preserves the radial symmetry of
the kernel \( \rho \), a property that will play a role in the sequel.
	\begin{Lema}\label{Radial} Assume $\rho$ verifies (\ref{eqn:H0}). Then, the kernel $K_\rho$ is radial.
	\end{Lema}
	\begin{proof} We just show that $K_\rho$ is rotation invariant. Let $R\in O(n)$ a rotation. Then, 
		\begin{align*}
			K_\rho(Rx)&=-\int_{\R^n} \frac{z}{|z|} \rho(z) \cdot\frac{z-Rx}{|z-Rx|} \rho(z-Rx)\,dz\\
			& \left[ \begin{array}{c} z=Ry \\
				\det (R)=1\end{array}\right] =-\int_{\R^n} \frac{Ry}{|Ry|} \rho(z) \frac{Ry-Rx}{|Ry-Rx|} \cdot\rho(Ry-Rx)\,dy\\
			&=-\int_{\R^n} \frac{y}{|y|} \rho(z) \cdot\frac{y-x}{|y-x|} \rho(y-x)\,dy =K_\rho(x).
		\end{align*}
	\end{proof}



	\begin{Obses}
	\end{Obses}
	\begin{enumerate}
		\item  If we choose $\rho=\rho^s$, we have that $-\Delta_\rho=(-\Delta)^s$, the well known fractional Laplacian. Indeed, since $Q_{\rho^s}=I_{1-s}$ and $\widehat{I_{1-s}u}=|2\pi\xi|^{s-1}\widehat{u}(\xi),$  we have that 
		$$\widehat{-\Delta_{\rho^s}u}=4\pi^2 |\xi|^2 |2\pi\xi|^{2s-2}\widehat{u}(\xi)=(2\pi|\xi|)^{2s}\widehat{u}(\xi)=\widehat{(-\Delta)^su}.$$ Moreover, $$K_{\rho^s}(x)=C_{n,s}\frac{1}{|x|^{n+2s}},$$ where $C_{n,s}$ is a suitable normalization constant.
		\item If we consider a rescaled kernel $\rho_\delta$, we get that the $\rho_\delta$-Laplacian, i.e., the  $\rho$-Laplacian of finite horizon $\delta>0$, is given by the Fourier multiplier $$
		\widehat{\Delta_{\rho_\delta} u}(\xi)=-4\pi^2|\xi|^2\widehat{Q}_{\rho_\delta}(\xi)^2\widehat{u}(\xi)=-4\pi^2c^2_\delta\delta^{-2n}|\xi|^2\widehat{Q_{\rho}}^2(\delta\xi)\widehat{u}(\xi).$$ In particular, for the choice $c_\delta=\delta^{-n}$, 
		$$\widehat{\Delta_{\rho_\delta} u}(\xi)=-4\pi^2|\xi|^2\widehat{Q_{\rho}}^2(\delta\xi)\widehat{u}(\xi),$$ i.e., the symbol $m_{K_{\rho_\delta}}(\xi)$ of $\Delta_{\rho_\delta}$ is $m_{K_\rho}(\delta\xi)$. Moreover, by Lemma \ref{Loc1}, $\Delta_{\rho_\delta}u$ converges to $\Delta u$ in the $L^p$-norm as $\delta\to 0^+$.
	\end{enumerate}

	\subsection{Nonlocal linear Sobolev spaces}
Henceforth, we  assume that the kernel is nonnegative, 
\begin{equation}\label{eq:K positive} K_\rho(x)\ge 0 \text{  for any  }x\ne 0. 
\end{equation}
Our aim now is to relate the \( \rho \)-Laplacian to the nonlocal Sobolev
spaces introduced in Definition \ref{def-non-local-soblev}, in the particular
case \( p = 2 \). A key observation in this direction is that the kernel
\( K_\rho \) naturally induces a bilinear form associated with the
\( \rho \)-Laplacian, defined on
\( C_c^\infty(\mathbb{R}^n) \times C_c^\infty(\mathbb{R}^n) \) by
\begin{equation}\label{eq-inner-product}
	\langle u, v \rangle_{\rho}
	:= \int_{\mathbb{R}^n}\int_{\mathbb{R}^n}
	\bigl(u(x)-u(y)\bigr)\bigl(v(x)-v(y)\bigr)
	K_\rho(x-y)\,dy\,dx.
\end{equation}
Provided \eqref{eq:K positive} an application of Hölder's inequality
immediately yields
\[
|\langle u, v \rangle_\rho|
\leq
\langle u, u \rangle_\rho^{1/2}
\langle v, v \rangle_\rho^{1/2},
\]
whenever both seminorms \( \langle u, u \rangle_\rho \) and
\( \langle v, v \rangle_\rho \) are finite. This naturally motivates
the introduction of a family of Gagliardo--Sobolev type spaces associated
with the kernel \( K_\rho \).
	\begin{Defi}[Nonlocal Gagliardo Spaces] \label{def-nonlocal-sobolev}
		Let $\rho$ a general kernel satisfying \eqref{eqn:H0}. We define the \textit{Nonlocal Gagliardo space} associated to $\rho$ as $$W^{\rho,2}(\R^n)=\{u\in L^2(\R^n): [u]_{\rho,2}<\infty\},$$ where $$[u]_{\rho,2}:=\int_{\R^n}\int_{\R^n}|u(x)-u(y)|^2 K_{\rho}(x-y)\,dy\,dx=\langle u,u\rangle_\rho,$$ endowed with the norm $$\norm{u}_{W^{\rho,2}}:=\left(\norm{u}_{2}^2+[u]_{\rho,2}\right)^{1/2}.$$
	\end{Defi}
	
	\begin{Obse}
		Note that when we choose $\rho=\rho^s=|x|^{-(n+s-1)}$, the seminorm $[u]_{\rho,2}$ becomes $$[u]_{\rho,2}=\int_{\R^n}\int_{\R^n}\frac{|u(x)-u(y)|^2}{|x-y|^{n+2s}}\,dy\,dx,$$ which is the well known Gagliardo seminorm. Hence, from general nonlocal gradients we can recover the Sobolev–Slobodecki spaces $W^{s,2}(\R^n)$. It is a well known fact that $W^{s,2}(\R^n)=H^{s,2}(\R^n)$, i.e., $W^{\rho^s,2}(\R^n)=H^{\rho^s,2}(\R^n)$, and hence it is natural to wonder if such a result hold for more general kernels $\rho$. In order to check this we need a potential characterization of the space $H^{\rho,2}(\R^n)$.
	\end{Obse}
	\begin{Prop}[Potential characterization of linear nonlocal spaces]\label{Potential}
		Let $\rho$ satisfy \eqref{eqn:H0} and let $K_\rho$  satisfy \eqref{eq:K positive}. Then, 
		$$H^{\rho,2}(\R^n)=\{u\in \mathcal{S}'(\R^n): \langle\lambda_\rho\rangle\widehat{u}(\xi)\in L^2\},$$  where $\langle\lambda_\rho\rangle:=(1+|\lambda_\rho(\xi)|^2)^{1/2}$, with $\lambda_\rho(\xi)=-2\pi i \xi \widehat{Q_\rho}(\xi)$.
	\end{Prop}
	\begin{proof}
	
	It is enough to notice that by Plancherel's identity, $$\int_{\R^n}|D_\rho u(x)|^2\,dx=\int_{\R^n}|\widehat{D_\rho u}(\xi)|^2\,d\xi=\int_{\R^n}|\lambda_\rho (\xi)|^2|\widehat{u}(\xi)|^2\,d\xi.$$ \end{proof} 
	\begin{Obse}
		Note that this is not an isolated property for the linear case. In fact, if we further assume that $\rho$ satisfies the properties \eqref{eqn:H1}, \eqref{eqn:H2}, by means of Milhin-H\"ormander's multiplier theorem it can be proven that $$H^{\rho,p}(\R^n)=\{u\in \mathcal{S}'(\R^n): \mathfrak{F}^{-1}\left(\langle\lambda_\rho\rangle\widehat{u}\right)\in L^p(\R^n)\},\,1<p<\infty.$$
	\end{Obse}
	Once we have characterized nonlocal Sobolev spaces as potential spaces we are ready to establish the equality between $W^{\rho,2}(\R^n)$ and $H^{\rho,2}(\R^n)$.
	\begin{Teor}\label{Igualdad}
		Let $\rho$ satisfy \eqref{eqn:H0} and let $K_\rho$  satisfy \eqref{eq:K positive}. Then, we have that $W^{\rho,2}(\R^n)=H^{\rho,2}(\R^n)$, with equivalence of the norms. 
	\end{Teor}
\begin{proof}
 Let $u\in \mathcal{S}'(\R^n)$. Then,
	\begin{align*}
		[u]_{\rho,2}&=\int_{\R^n}\int_{\R^n} |u(x+h)-u(x)|^2K_\rho(h)\,dh\,dx
		\\&=\int_{\R^n}\left(\int_{\R^n} |u(x+h)-u(x)|^2\,dx\right)K_\rho(h)\,dh\\
		&=\int_{\R^n} \norm{u(\cdot+h)-u(\cdot)}_2^2K_\rho(h)\,dh=\int_{\R^n} \norm{\mathfrak{F}\left((u(\cdot+h)-u(\cdot)\right)}_2^2K_\rho(h)\,dh\\
		&=\int_{\R^n} \int_{\R^n} |e^{2\pi i \xi\cdot h}-1|^2|\widehat{u}(\xi)|^2K_\rho(h)\,d\xi\,dh\\
		&=2\int_{\R^n} \int_{\R^n} (1-\cos (2\pi \xi\cdot h))|\widehat{u}(\xi)|^2 K_\rho(h)\,dh\,d\xi.
	\end{align*}
	By the radial symmetry of $K_\rho$ from Lemma \ref{Radial}, we have that we can compute the integral
	\begin{align*}\int_{\R^n} \left(1-e^{-2\pi i \xi\cdot h}\right) K_\rho(h)\,dh&=\int_{\R^n} (1-\cos(2\pi \xi\cdot h))K_\rho(h)\,dh \\
    &=\int_{\R^n}K_\rho(h)\,dh-\int_{\R^n}\cos(2\pi\xi\cdot h)K_\rho(h)\,dh\\
		&=\widehat{{K_\rho}}(0)-\widehat{{K_\rho}}(\xi)=-\widehat{{K_\rho}}(\xi),\end{align*} as minus the Fourier symbol $m_{K_\rho}$ of our operator $K_\rho$, i.e., $|\lambda_\rho (\xi)|^2$. Hence, \begin{align*}[u]_{\rho,2}&=2\int_{\R^n}|\lambda_\rho(\xi)|^2|\widehat{u}(\xi)|^2\,d\xi\implies \norm{u}_2^2+\frac{1}{2}[u]_{\rho,2}=\int_{\R^n}\left(1+|\lambda_\rho(\xi)|^2\right)|\widehat{u}(\xi)|^2\,d\xi\\
		&=\norm{\langle \lambda_\rho \rangle \widehat{u}}_2^2,\end{align*} so the spaces $H^{\rho,2}(\R^n)=W^{K,2}(\R^n)$ coincide with equivalence of the norms by Proposition \ref{Potential}. 
	\end{proof}
\begin{Obses}
\end{Obses}	
\begin{itemize}
    \item As a consequence of the preceding results, the bilinear form
\( \langle u, v \rangle_\rho \) is well defined for every
\( u, v \in H^{\rho,2}(\mathbb{R}^n) \).
\item In view of the last result, we could extend the definition of Nonlocal Gagliardo Spaces for general kernels to $1<p<\infty$ as the spaces of $L^p(\R^n)$ functions $u$ such as $$[u]_{\rho,p}:=\int_{\R^n\times \R^n}|u(x)-u(y)|^pK_{\rho}(x-y)|x-y|^{-(p-2)s}\,dx\,dy<\infty.$$ It would be interesting to study in a subsequent work the relantionship between $W^{\rho,p}(\R^n)$ and $H^{\rho,p}(\R^n)$.
\end{itemize}

We next establish the continuous
embedding
\[
W^{1,2}(\mathbb{R}^n) \hookrightarrow H^{\rho,2}(\mathbb{R}^n),
\]
which was originally proved in greater generality in
\cite[Proposition~3.5]{BellidoMoraHidde2025}; we include the proof here
for the sake of completeness.
	\begin{Lema}\label{Inclusion}
		Let $\rho$ satisfying \eqref{eqn:H0}. Then, there exists a positive constant $C=C(n,\rho)$ such that $$\norm{u}_{H^{\rho,2}(\R^n)}\leq C\norm{u}_{W^{1,2}(\R^n)},$$
        for any $u \in W^{1,2}(\R^n)$. 
	\end{Lema}
	\begin{proof}Let $u\in W^{1,2}(\R^n)$. Then, $$\norm{D_\rho u}_2\leq \norm{\int_{B_1}\frac{|u(\cdot+h)-u(\cdot)}{|h|}\rho(h)\,dh}_2+\norm{\int_{B_1^c}\frac{|u(\cdot+h)-u(\cdot)}{|h|}\rho(h)\,dh}_2.$$ By Minkowski's integral inequality, \begin{align*}
		\norm{\int_{B_1}\frac{|u(\cdot+h)-u(\cdot)|}{|h|}\rho(h)\,dh}_2\leq \int_{B_1}\frac{\rho(h)}{|h|}\norm{u(\cdot+h)-u(\cdot)|}_2\,dh\leq C\norm{D u}_2\int_{B_1}\rho(h)\,dh.
	\end{align*} For the other integral we can use Fubini's to get that \begin{align*}
		&\norm{\int_{B_1^c}\frac{|u(\cdot+h)-u(\cdot)|}{|h|}\rho(h)\,dh}_2\leq \norm{\int_{B_1^c}\frac{|u(\cdot+h)|}{|h|}\rho(h)\,dh}_2+\norm{\int_{B_1^c}\frac{|u(\cdot)|}{|h|}\rho(h)\,dh}_2\\
		& \leq 2\norm{u}_p\int_{B_1^c}\rho(h)|h|^{-1}\,dh.
	\end{align*}Puting all togheter we get that $$\norm{D_\rho u}\leq C\norm{u}_{W^{1,2}}\int_{\R^n}\min\{1,|h|^{-1}\}\rho(h)\,dh\leq C\norm{u}_{W^{1,2}},$$ by the hypothesis on $\rho$. 
\end{proof}

\subsection{Relationship with integro-differential equations}
The relationship between our family of kernels $K_\rho$ derived from general nonlocal kernels $\rho$, and the class $\mathcal{K}_s(\lambda,\Lambda)$ comes from the, although simple, important observation that $-\Delta_\rho$ is an integro-differential operator in the sense of \cite{FernandezRos}.
	
	\begin{Prop} \label{prop-rho-laplaciano} For any $u\in\mathcal{S}$, the following identity holds
		\begin{equation}\label{eq-non-local-laplacian}
			\mathcal{L}_\rho u(x)=\int_{\R^n} (u(x)-u(x+h)) K_\rho(h)\,dh=-\Delta_\rho u(x).
		\end{equation}
	\end{Prop}
	\begin{proof} It is enough to show that 
		\[\widehat{{\mathcal{L}_\rho} {u}}(\xi)=4\pi^2|\xi|^2\widehat{Q}_\rho(\xi)^2 \widehat{u}(\xi),\]
		for any $u\in \mathcal{S}$. But this is straightforward,
		\begin{align*}
			\widehat{{\mathcal{L}_\rho} {u}}(\xi)&= \int_{\R^n}(1-e^{-2\pi i \xi\cdot h})\widehat{u}(\xi)K_\rho(h)\,dh\\
			&= \left(\int_{\R^n}(1-e^{-2\pi i \xi\cdot h})K_\rho(h)\,dh\right) \widehat{u}(\xi)\\
			&= (\widehat{K_\rho}(0)-\widehat{K_\rho}(\xi))\widehat{u}(\xi)=-\widehat{K_\rho}(\xi)\widehat{u}(\xi)\\
			&= 4\pi^2 |\xi|^2 \widehat{Q}_\rho(\xi)^2 \widehat{u}(\xi).
		\end{align*}
	\end{proof}

	As a consequence of this embedding and the equivalence between the spaces from Theorem \ref{Igualdad}, we get that our nonlocal operators $K_\rho$ satisfies the integrability condition of the class $\mathcal{K}_s(\lambda,\Lambda)$ \eqref{eq:integrabilityK}, i.e., 
	\[
		\int_{\R^n}\min\{1,|y|^2\}K_\rho (y)\,dy<\infty,
	\]
	as a direct consequence of the theory developed in \cite{Foghem2025}.
	\begin{Lema}
		\label{lemma: int condition}
		Let $\rho$ satisfy \eqref{eqn:H0} and let $K_\rho$  satisfy \eqref{eq:K positive}. Then, the operator $K_\rho$ verifies the condition \eqref{eq:integrabilityK}.
	\end{Lema}
	\begin{proof}
		 By slightly adapting the notation, by \cite[Theorem 4.4]{Foghem2025} we have that $K_\rho$ satisfies the condition \ref{eq:integrabilityK} if and only if the continuous embedding $W^{1,2}(\R^n)\xhookrightarrow{}W^{\rho,2}(\R^n)$ holds. Hence, the result follows from Theorem \ref{Igualdad} and Lemma \ref{Inclusion}. 
	\end{proof}
	
	In view of this result, we can easily adapt the proof of \cite[Lemma 2.2.16]{FernandezRos} combined with the integration by parts formula for $D_\rho$ from Lemma \ref{IBP} to obtain the following integration by parts formula for $\mathcal{L}_\rho$:
	\begin{Lema} \label{lemma-inner-produc-formu}
		Let $\rho$ satisfy \eqref{eqn:H0} and let $K_\rho$  satisfy \eqref{eq:K positive}. Then, for every $u,v\in H^{\rho,2}(\R^n)$ we have that $$\langle D_\rho u,D_\rho v\rangle_{L^2(\R^n)}=\int_{\R^n}u(\mathcal{L}_\rho v)\,dx=\langle u,v\rangle_{\rho}=\int_{\R^n}(\mathcal{L}_\rho u)v\,dx.$$
	\end{Lema}
	
In order to compare the Fourier symbol of the nonlocal operator \( K_\rho \)
with that of the fractional Laplacian, additional conditions on the kernel
\( \rho \) are required. Specifically, we need the nonlocal kernel \( \rho \)
to be controlled near the origin by two purely fractional kernels, which is
encoded in hypotheses \eqref{eqn:H3}--\eqref{eqn:H4}. We further require
the positivity of \( \widehat{Q}_\rho \), guaranteed by hypothesis
\eqref{eqn:H1}. Under these conditions, the following estimates on
\( \widehat{Q}_\rho(\xi) \) hold (see \cite[Lemma~2.6]{CuetoHidde2025}).
	\begin{Lema}[Estimates on $\widehat{Q}_\rho$]
		Let $\rho$ satisfying \eqref{eqn:H0} and \eqref{eqn:H1} and let $K_\rho$ satisfying \eqref{eq:K positive}. Then, there exists a positive constant $C=C(n,\rho)$ such that $$\frac{1}{C}\frac{\overline{\rho}(1/|\xi|)}{|\xi|^n}\leq \widehat{Q}_\rho(\xi)\leq C\frac{\overline{\rho}(1/|\xi|)}{|\xi|^n},\,|\xi|\geq 1/\varepsilon.$$
	\end{Lema}
	If we further assume \eqref{eqn:H3} and \eqref{eqn:H4}, we get as a direct consequence the following result:
	\begin{Lema}[Fractional estimates on $\widehat{Q}_\rho$]
		Let $\rho$ satisfying \eqref{eqn:H0},\eqref{eqn:H1},\eqref{eqn:H3} and \eqref{eqn:H4}and let $K_\rho$ satisfying \eqref{eq:K positive}. Then, there exists a positive constant $C=C(n,\rho)$ such that $$\frac{1}{C}|\xi|^{s-1}\leq \widehat{Q}_\rho(\xi)\leq C|\xi|^{t-1},\,|\xi|\geq 1/\varepsilon,$$ where $s\leq t$ are the fractional parameters on the hypothesis \eqref{eqn:H3} and \eqref{eqn:H4}. 
	\end{Lema}
	From here, the following estimates for the multiplier $m_{K_\rho}$ are straightforward.
	\begin{Coro}\label{Multiplier}
		Let $\rho$ satisfying \eqref{eqn:H0},\eqref{eqn:H1},\eqref{eqn:H3} and \eqref{eqn:H4} with $s=t$, and let $K_\rho$ satisfying \eqref{eq:K positive}. Then, there exists a positive constant $C=C(n,\rho)$ such that $$\frac{4\pi^2}{C^2}|\xi|^{2s}\leq m_{K_\rho}(\xi)\leq 4\pi^2C^2|\xi|^{2t},\,\forall|\xi|\geq 1/\varepsilon.$$
	\end{Coro}
	\begin{proof}
	Remember that $\lambda_\rho(\xi)=2\pi i\xi\widehat{Q}_\rho(\xi)$ and $m_{K_\rho}(\xi)=|\lambda_\rho(\xi)|^2$. Hence, by the previous Lemma, \begin{align*}
		m_{K_\rho}(\xi)&=|\lambda_\rho(\xi)|^2=4\pi^2|\xi|^2\widehat{Q}_\rho^2(\xi)\geq \frac{4\pi^2}{C^2}|\xi|^2|\xi|^{2s-2}=\frac{4\pi^2} {C^2}|\xi|^{2s},\\
		m_{K_\rho}(\xi)&=|\lambda_\rho(\xi)|^2=4\pi^2|\xi|^2\widehat{Q}_\rho^2(\xi)\leq  4\pi^2C^2|\xi|^2|\xi|^{2t-2}=4\pi^2 C^2|\xi|^{2t},
	\end{align*} as we wanted to prove. 
\end{proof}

Observe that for any \( u, v \in H^{\rho,2}(\mathbb{R}^n) \), Plancherel's
theorem yields
\[\langle u, v \rangle_\rho=\int_{\mathbb{R}^n} |\lambda_\rho(\xi)|^2\widehat{u}(\xi)\,\widehat{v}(\xi)\,d\xi.\]
Moreover, if \( \rho \) satisfies \eqref{eqn:H0}-\eqref{eqn:H4} with
\( s = t \), the preceding lemma gives
\[[u]_{\rho,2} = \langle u, u \rangle_\rho \approx \|D^s u\|_{2}^2,\]
and hence the bilinear form \( \langle \cdot, \cdot \rangle_\rho \) extends
naturally to \( H^{s,2}(\mathbb{R}^n) \times H^{s,2}(\mathbb{R}^n) \).
As a conclusion for this section, we obtain that our family of nonlocal kernels $K_\rho$ defines integro-differential operators included in the class considered \cite{FernandezRos} under suitable hypothesis.
\begin{Teor}
    Let $\rho$ a nonlocal kernel satisfying (\ref{eqn:H0})--(\ref{eqn:H4}) with $s=t$. Then, $K_\rho \in \mathcal{K}_s(\lambda,\Lambda)$.
\end{Teor}
	
\section{Maximum principle}
The functional framework developed so far provides the necessary tools to
study the Dirichlet problem associated with the \( \rho \)-Laplacian, namely
\begin{equation}\label{eqn:Dirichlet}
	\begin{cases}
		(-\Delta)_\rho u = f, & x \in \Omega, \\
		u = 0, & x \in \Omega^c.
	\end{cases}
\end{equation}
Thanks to the Poincaré inequality \eqref{Poincare-ineq}, the above problem
admits the equivalent variational formulation
\[
\min_{u \in H_0^{\rho,2}(\Omega)}
\left\{
\frac{1}{2}
\int_{\mathbb{R}^n} |D_\rho u|^2\,dx
-
\int_\Omega f u\,dx
\right\}.
\]
Furthermore, the results of the preceding sections allow us to express the
corresponding weak formulation in two equivalent ways. On the one hand,
using the nonlocal gradient, a function \( u \in H_0^{\rho,2}(\Omega) \)
is a weak solution if
\[
\int_{\mathbb{R}^n}
D_\rho u \cdot D_\rho \phi\,dx
=
\int_\Omega f\phi\,dx
\qquad
\text{for every } \phi \in C_c^\infty(\mathbb{R}^n).
\]
On the other hand, by Lemma~\ref{lemma-inner-produc-formu}, this is
equivalently expressed in terms of the bilinear form \eqref{eq-inner-product}
as
\[
\langle u, \phi \rangle_\rho = \int_\Omega f\phi\,dx
\qquad
\text{for every } \phi \in C_c^\infty(\mathbb{R}^n).
\]
This dual weak formulation provides a natural starting point for the study
of more general nonlocal elliptic equations, encompassing lower-order terms
and anisotropic effects, and thereby extending the framework associated with
the Riesz fractional gradient studied in \cite{ShiehSpector2015,ShiehSpector2018}.
\[
\begin{cases}
	-\operatorname{div}_{\rho}\bigl(A(x)D_\rho u\bigr) + b(x)u = f,
	& x \in \Omega, \\
	u = 0,
	& x \in \Omega^c,
\end{cases}
\]
for which existence and uniqueness of weak solutions follow, under suitable
assumptions on the coefficients \( A \) and \( b \), as a direct application
of the Lax-Milgram theorem.

	Due to the variational nature of nonlocal gradients, we can further define a nonlocal $p$-Laplacian for a general kernel $\rho$ as the first variation of the $p$-norm of $D_\rho u$ for $u\in H^{\rho,p}(\R^n)$. We introduce the operator $$\Delta_{\rho,p}:H^{\rho,p}_0(\Omega)\to \left(H^{\rho,p}_0(\Omega)\right)^*,\,u\mapsto \operatorname{div}_\rho\left(|D_\rho u|^{p-2}D_\rho u|\right),$$ as the $(\rho,p)$-\textit{Laplacian}, or simply \textit{nonlocal $p$-Laplacian}. For every $u\in H_0^{\rho,p}(\Omega)$, we define the action of $\Delta_{\rho,p}u$ over $v\in H_0^{\rho,p}(\Omega)$ as $$\langle-\Delta_{\rho,p}u,v\rangle:=\int_{\R^n}|D_\rho u|^{p-2}D_\rho u\cdot D_\rho v\,dx,$$  which is well defined by the nonlocal integration by parts formula. For $p=2$, the $(\rho,p)$-Laplacian is just the $\rho$-Laplacian, and for the choice $\rho=\rho^s$, we get the fractional $p$-Laplacian introduced in \cite{Schikorra2018} as $H^{s,p}$-Laplacian, which coincides for $p=2$ with the usual fractional $p$-Laplacian arising from the Gagliardo seminorm.  Clearly, the formula of $\Delta_{\rho,p}$ completely resembles the classical $p$-Laplacian $-\operatorname{div}(|D u|^{p-2}D u)$.
	Let us briefly look at the following Dirichlet problem \begin{equation}\label{eqn:DirichletP}
		\begin{cases}
			-\Delta_{\rho,p}u&=f,\,x\in \Omega,\\
			\hfill u&=0,\,x\in \Omega^c
		\end{cases}    
		.\end{equation} This problem is a paradigmatic example of a singular quasilinear elliptic problem on the nonlocal framework. Existence of weak solutions for \eqref{eqn:DirichletP} is a consequence of the abstract existence result \cite[Theorem 2.11]{CuetoHidde2025} applied to the energy density $W(D_\rho u)=|D_\rho u|^p$. Indeed, completely analogous as the case of the classical $p$-Laplacian, we can derive that $\Delta_{\rho,p}$ is an strongly monotone operator, yielding the uniqueness of weak solutions for the problem \eqref{eqn:DirichletP}.
	\begin{Teor}
		Let $\Omega\subset \R^n$ a bounded domain, $1<p,q<\infty$  such that $1/p+1/q=1$, $f\in L^q(\Omega)$ and $\rho$ a compactly supported kernel satisfying the hypothesis on Theorem \ref{Poincare-ineq}. Then, there exists an unique $u\in H_0^{\rho,p}$ such that $$\int_{\R^n}|D_\rho u|^{p-2}D_\rho u\cdot D_\rho v\,dx=\int_\Omega f v\,dx,\,\forall v\in C_c^\infty(\Omega).$$
	\end{Teor}
	For the choice $\rho=\rho^s$, in \cite{Schikorra2018}, the authors established interior H\"older regularity for solutions of \eqref{eqn:DirichletP} in the homogeneus case. More recently, in \cite{Camilo2026}, on of the the authors proved global regularity for the solutions assuming that $f$ is in a suitable Besov space. For the homogeneus case,  the arguments in \cite{Schikorra2018} could be easily adapted. We leave for a subsequent work the study of the regularity in the spirit of  \cite{FernandezRos} for our family of nonlocal operators.  
	\subsection{Maximum principle}
	
	In the classical theory of elliptic equations, maximum principles are a direct consequence of the sign structure of the operator at extremal points. For nonlocal operators, the mechanism is similar in spirit, although it reflects the genuinely nonlocal nature of the equation. Indeed, if a function \(u\) attains a global maximum at some point \(x_{0}\), then
	\[
	u(x_{0})-u(x_{0}+y)\geq 0
	\qquad \text{for every } y\in\mathbb{R}^{n}.
	\]
	Hence, whenever the kernel \(K_{\rho}\) is nonnegative, the integral representation of the operator immediately provides the appropriate sign of \(\mathcal{L}u(x_{0})\).
	
	A fundamental difference with respect to local operators is that the value of \(\mathcal{L}u(x_{0})\) depends on the behavior of \(u\) in the whole space. Consequently, information at a merely local extremum is not sufficient, and one must work with global maxima or minima. This nonlocal interaction is precisely the feature that allows extremal values to propagate through the support of the kernel.
	
	In order to make the previous argument rigorous, one needs enough regularity to evaluate \(\mathcal{L}_\rho u(x)\) pointwise and to ensure that the defining integral is well posed. By Lemma \ref{lemma: int condition}, the kernel \(K_\rho\) satisfies \eqref{eq:integrabilityK}, and therefore \(K_\rho \in L^1_{\mathrm{loc}}(\mathbb{R}^n\setminus\{0\})\). Under these assumptions, the proof follows along the lines of the classical argument for nonlocal integro-differential operators (see, e.g., \cite{RoSe14}). We thus obtain the following strong maximum principle.
	\begin{Prop}[Maximum principle]
		 Let $\rho$ satisfy \eqref{eqn:H0} and let $K_\rho$ satisfy \eqref{eq:K positive},
		and let \( \mathcal{L} \) be the operator introduced in
		\eqref{eq-non-local-laplacian}. Let \( \Omega \subset \mathbb{R}^n \) be a
		bounded open set, and let \( u \in C_{\mathrm{loc}}^{1,1}(\mathbb{R}^n)
		\cap L^\infty(\mathbb{R}^n) \). If
		\[
		\begin{cases}
			\mathcal{L}u \geq 0 & \text{in } \Omega, \\
			u \geq 0 & \text{in } \Omega^c,
		\end{cases}
		\]
		then \( u \geq 0 \) in \( \mathbb{R}^n \).
	\end{Prop}
	
\begin{proof}
	Assume by contradiction that \(m:=\inf_{\mathbb{R}^{n}}u<0.\) Since \(u\in C(\overline{\Omega})\) and \(u\geq 0\) in \(\Omega^c\), there exists some \(x_{0}\in\Omega\) such that \(u(x_{0})=m.\) Because \(x_{0}\) is a global minimum point of \(u\), we have
	\(u(x_{0})-u(x_{0}+h)\leq 0\)
	for every \(h\in\mathbb{R}^{n}.\)
	
	Since \(K_{\rho}(h)\geq 0\), it follows that
	\[
	\mathcal{L}u(x_{0})
	=
	\int_{\mathbb{R}^{n}}
	\bigl(u(x_{0})-u(x_{0}+h)\bigr)K_{\rho}(h)\,dh
	\leq 0.
	\]
	On the other hand, by hypothesis, \(\mathcal{L}u(x_{0})\geq 0.\) Therefore,
	\(\mathcal{L}u(x_{0})=0.\)
	
	Since \(K_\rho\) is nontrivial and belongs to \(L^1_{\mathrm{loc}}(\mathbb{R}^n\setminus\{0\})\), there exist \(a\neq0\) and \(r>0\) such that
	\[
	\int_{B_r(a)}K_\rho(h)\,dh>0.
	\]
	If \(u(x_0+h)>u(x_0)\) for almost every \(h\in B_r(a)\) such that \(K_\rho(h)>0\), then
	\[
	\int_{B_r(a)}
	\bigl(u(x_0)-u(x_0+h)\bigr)
	K_\rho(h)\,dh<0,
	\]
	contradicting the identity \(\mathcal{L}u(x_0)=0\). Hence, there exists \(h_0\in B_r(a)\) such that \(u(x_0+h_0)=u(x_0)=m.\) Repeating the same argument at the point \(x_1:=x_0+h_0\), we either obtain a contradiction or find \(x_2\in x_1+B_r(a)\) such that \(u(x_2)=m\). Proceeding inductively, we construct a sequence \(\{x_k\}\) satisfying
	\[
	x_{k+1}\in x_k+B_r(a),
	\qquad
	u(x_k)=m.
	\]
	Since \(\Omega\) is bounded, after finitely many steps we either obtain a point $x_k\in \Omega$, such that  $\mathcal{L}u(x_k)<0$ or $x_k\not \in \Omega$, contradicting the assumption \(u\ge0\) in \(\Omega^c\). In both cases we end up in a contradiction. Therefore,
	\(u\ge0\) in \(\mathbb{R}^n\).
\end{proof}
	As a direct consequence of the previous maximum principle, one immediately obtains the following comparison principle.
	
	\begin{Coro}[Comparison principle]
		Let $\rho$ satisfy \eqref{eqn:H0} and let $K_\rho$ satisfy \eqref{eq:K positive},
		and let \( \mathcal{L} \) be the operator introduced in
		\eqref{eq-non-local-laplacian}. Let \( \Omega \subset \mathbb{R}^n \)
		be a bounded open set, and let \( u, v \in C_{\mathrm{loc}}^{1,1}(\mathbb{R}^n)
		\cap L^\infty(\mathbb{R}^n) \). 
		If
		\[
		\begin{cases}
			\mathcal{L}u \geq \mathcal{L}v & \text{in } \Omega, \\
			u \geq v & \text{in } \Omega^c,
		\end{cases}
		\]
		then \( u \geq v \) in \( \mathbb{R}^n \).
	\end{Coro}
	\begin{proof}
		Define \(w:=u-v.\) Since \(\mathcal{L}\) is linear, we have \(\mathcal{L}w=\mathcal{L}u-\mathcal{L}v\geq 0\) in \(\Omega.\) Moreover, \(w=u-v\geq 0\)
		in  \(\Omega^c.\) Therefore, the previous maximum principle applied to \(w\) yields \(w\geq 0\) in \(\mathbb{R}^{n},\) that is, \(u\geq v \) in \(\mathbb{R}^{n}.\)
	\end{proof}
	
	\subsection{Weak maximum principle}
	
The positivity of the kernel \( K_\rho \) enables us to establish a weak
maximum principle for solutions of the Dirichlet problem \eqref{eqn:Dirichlet}.
The key observation is that the weak formulation
\[
\langle u, \phi \rangle_\rho = \int_{\mathbb{R}^n} f(x)\phi(x)\,dx
\]
is completely analogous to that of the fractional Laplacian, and the
hypothesis \eqref{eq:K positive} on \( K_\rho \) allows one to use the positive and negative
parts of a solution as test functions. Following the classical argument
of \cite{ServadeiValdinoci2014}, we obtain the
following result.

\begin{Prop}[Weak maximum principle]
	Let $\rho$ satisfy \eqref{eqn:H0} and let $K_\rho$ satisfy \eqref{eq:K positive}. Let \( \Omega \subset \mathbb{R}^n \) be a bounded open set. If
	\( u \in H_0^{\rho,2}(\Omega) \) is a weak solution of \eqref{eqn:DirichletP}
	such that \( \langle u, \phi \rangle_\rho \geq 0 \) for every nonnegative
	\( \phi \in H_0^{\rho,2}(\Omega) \) and \( u \geq 0 \) in \( \Omega^c \),
	then \( u \geq 0 \) in \( \mathbb{R}^n \).
\end{Prop}
	
	\begin{proof}
		Let \( u=u^{+}-u^{-}, \) where \( u^{+}:=\max\{u,0\},\) and \( u^{-}:=\max\{-u,0\}.\) 	Since \(u\geq 0\) in \(\Omega^c\), it follows that \( u^{-}=0 \) in \( \Omega^c. \) Moreover, using that \( 	|a^{-}-b^{-}|\leq |a-b|, \) we deduce that  \( u^{-}\in H^{\rho,2}_0(\Omega). \) Therefore, we may take \(\varphi=u^{-}\) as a test function in the weak formulation, obtaining \(0\leq\langle u,u^{-}\rangle_{\rho}.\) Using the decomposition \(u=u^{+}-u^{-}\), we obtain
		
		\[\langle u,u^{-}\rangle_{\rho}=\langle u^{+},u^{-}\rangle_{\rho}-\langle u^{-},u^{-}\rangle_{\rho}.\]
		
		Now, since \( u^{+}(x)u^{-}(x)=0 \) for a.e \( x \in \R^{n}, \) one has\(\left(u^{+}(x)-u^{+}(y)\right)\left(u^{-}(x)-u^{-}(y)\right)\leq 0 \) for a.e. \(x,y\in\mathbb{R}^{n}\). Since \(K_{\rho}\geq 0\), it follows that \(\langle u^{+},u^{-}\rangle_{\rho}\leq0.\) Consequently,
		
		\[0\leq- \langle u^{-},u^{-}\rangle_{\rho}.\]
		
		On the other hand, by positivity of the kernel, \(\langle u^{-},u^{-}\rangle_{\rho}\geq 0.\) Hence, \(\langle u^{-},u^{-}\rangle_{\rho}=0.\)
		Therefore, \( u^{-}\equiv 0 \)  in \( \R^n, \) which implies \( u\geq 0 \) in \( \R^n. \)
	\end{proof}
	
As a direct consequence of the weak maximum principle, we obtain the
following comparison principle.

\begin{Coro}[Comparison principle]
	Let $\rho$ satisfy \eqref{eqn:H0} and let $K_\rho$ satisfy \eqref{eq:K positive}. Let \( \Omega \subset \mathbb{R}^n \) be a bounded open set. If
	\( u, v \in H_0^{\rho,2}(\Omega) \) satisfy
	\( \langle u, \phi \rangle_\rho \geq \langle v, \phi \rangle_\rho \)
	for every nonnegative \( \phi \in H_0^{\rho,2}(\Omega) \) and \( u \geq v \)
	in \( \Omega^c \), then \( u \geq v \) in \( \mathbb{R}^n \).
\end{Coro}
	\begin{proof}
		Define \(w:=u-v.\) Therefore, the previous maximum principle applied to \(w\) yields \(w\geq 0\) in \(\mathbb{R}^{n},\) that is, \(u\geq v \) in \(\mathbb{R}^{n}.\)
	\end{proof}

   \addcontentsline{toc}{section}{References}
	\bibliographystyle{plain} 
	
	\addcontentsline{toc}{section}{Statements \& Declarations}
	\subsection*{Funding}
    J.C.B. is supported by {\it Agencia Estatal de Investigación} (Spain) through grant PID2023-151823NB-I00 and {\it Junta de Comunidades de Castilla-La Mancha} (Spain) through grant SBPLY/23/180225/000023. G.G.-S. is supported by a Doctoral Fellowship by Universidad de Castilla-La Mancha.  \text{2024-UNIVERS-12844-404}. J.C.R is  supported by ANII under the project ANII FCE 3 2024 1 181302.
	\subsection*{Conflicts of interest}
	The authors declare that there are no conflicts of interest regarding the publication of this paper.
	
\end{document}